\documentclass[12pt,a4paper]{amsart}
\usepackage{amsmath}
\usepackage{drpack}
\usepackage[english]{babel}
\usepackage{verbatim}
\usepackage[shortlabels]{enumitem}
\usepackage{color}
\usepackage{tikz-cd}

\newtheoremstyle{mio}%
{}{} 
{\itshape}{} 
{\bfseries}{.}{ } 
{#1 #2\thmnote{~\mdseries(#3)}} 
\theoremstyle{mio}
\newtheorem{teor}{Theorem}[section]
\newtheorem{cor}[teor]{Corollary}
\newtheorem{prop}[teor]{Proposition}
\newtheorem{lemma}[teor]{Lemma}
\newtheorem{defin}[teor]{Definition}

\newtheoremstyle{definition2}%
{}{} 
{}{} 
{\bfseries}{.}{ } 
{#1 #2\thmnote{\mdseries~ #3}} 
\theoremstyle{definition}

\title[Complete integral closure in Pr\"ufer domains]{The complete integral closure of a Pr\"ufer domain is a topological property}
\author{Dario Spirito}
\date{\today}
\address{Dipartimento di Scienze Matematiche, Informatiche e Fisiche, Universit\`a degli Studi di Udine, Udine, Italy}
\email{dario.spirito@uniud.it}
\subjclass[2010]{}

\newcommand{\marginparr}[1]{{\color{blue}{{\large$\bullet$}}}\marginpar{\footnotesize{\texttt{#1}}}}

\newcommand{\V}{\mathcal{V}}
\newcommand{\D}{\mathcal{D}}
\newcommand{\Int}{\mathrm{Int}}

\newcommand{\Min}{\mathrm{Min}}
\newcommand{\kn}{\mathrm{Kr}}
\newcommand{\qf}{\rho}

\begin{document}
\begin{abstract}
We show that the prime spectrum of the complete integral closure $D^\ast$ of a Pr\"ufer domain $D$ is completely determined by the Zariski topology on the spectrum $\Spec(D)$ of $D$.
\end{abstract}

\maketitle

\section{Introduction}
Let $D$ be an integral domain with quotient field $K$. An element $z\in K$ is \emph{almost integral} over $D$ if there is a nonzero element $c\in D$ such that $cz^n\in D$ for every positive integer $n$. The set $D^\ast$ of all the elements of $K$ that are almost integral over $D$ is a ring, called the \emph{complete integral closure} of $D$, and $D$ is said to be \emph{completely integrally closed} if $D=D^\ast$. However, the term ``closure'' is a misnomer, as $D^\ast$ need not to be completely integrally closed; the first example of this phenomenon was given by Gilmer and Heinzer \cite[Example 1]{gilmer_CIC_1966}, and it was shown later that this can happen also for Pr\"ufer domains \cite{heinzer-cic}, even in the finite-dimensional case \cite{sheldon-cic}. It is also possible that the process of taking repeatedly the complete integral closure never stabilizes \cite{hill_cic}.

The purpose of this paper is to show that, for a Pr\"ufer domain $D$, the complete integral closure $D^\ast$ (or rather, the spectrum of $D^\ast$) is uniquely determined by the spectrum $\Spec(D)$ of $D$ and by its Zariski topology. We do this by first reducing to the case of B\'ezout domains (through the use of the Kronecker function ring) and then by characterizing the condition $PD^\ast=D^\ast$ through the density of a particular enlargement of the closed cocompact sets containing $P$ (Theorem \ref{teor:Specast}). As a consequence, we prove that, for a Pr\"ufer domain, the condition of being completely integrally closed depends uniquely on $\Spec(D)$, and we prove a sufficient condition under which $D^\ast$ is completely integrally closed (Proposition \ref{prop:MinIcomp}). We also link the complete integral closure of $D$ with the sets of minimal primes of the finitely generated ideals in $D$ (Theorem \ref{teor:MinIMinJ}).

\bigskip

Throughout the paper, $D$ is an integral domain with quotient field $K$. A \emph{fractional ideal} of $D$ is a $D$-submodule $I$ of $K$ such that $dI\subseteq D$ for some $d\neq 0$. In particular, the fractional ideals of $D$ contained in $D$ are the ideals of $D$. For a fractional ideal $I$, we set $I^{-1}:=(D:I):=\{x\in K\mid xI\subseteq D\}$. A fractional ideal $I$ is \emph{invertible} if $IJ=D$ for some fractional ideal $J$; in this case, $J=I^{-1}$.

The \emph{spectrum} $\Spec(D)$ of $D$ is the set of its prime ideals; the Zariski topology is the topology whose closed sets are the sets in the form $\V(I):=\{P\in\Spec(D)\mid I\subseteq P\}$, where $I$ is an ideal of $D$. We say that a set $X\subseteq\Spec(D)$ is \emph{cocompact} in $\Spec(D)$ if $\Spec(D)\setminus X$ is compact; if $X$ is closed, then $X$ is cocompact if and only if $X=\D(I):=\Spec(D)\setminus\V(I)$ for some finitely generated ideal $I$.

If $I$ is a proper ideal of $D$, a \emph{minimal prime} of $I$ is a prime ideal that is minimal among the prime ideals containing $I$, i.e., a minimal element of $\V(I)$. We denote by $\Min(I)$ the set of minimal primes of $I$.

A \emph{valuation domain} is an integral domain $D$ such that, for every $x\in K$, at least one of $x$ and $x^{-1}$ belongs to $D$; a valuation domain is local, and the set of fractional ideals of a valuation domain is totally ordered. A \emph{Pr\"ufer domain} is an integral domain such that $D_P$ is a valuation domain for every $P\in\Spec(D)$. See \cite[Chapter IV]{gilmer} and \cite[Theorem 1.1.1]{fontana_libro} for many characterization of Pr\"ufer domains. A \emph{B\'ezout domain} is an integral domain such that every finitely generated ideal is principal; every B\'ezout domain is Pr\"ufer.

The \emph{Kronecker function ring} $\kn(D)$ of $D$ is the set of all rational functions $f(X)/g(X)$ in one indeterminate $X$ such that $c(f)V\subseteq c(g)V$ for all valuation overrings $V$ of $D$ (where $c(f)$ is the \emph{content} of $f$, i.e., the ideal generated by the coefficients of $D$). The ring $\kn(D)$ is a B\'ezout domain. If $D$ is a Pr\"ufer domain, the assignment $P\mapsto P\kn(D)$ is a homeomorphism from $\Spec(D)$ to $\Spec(\kn(D))$ \cite{fontana_krr-abRs}, and the rings between $\kn(D)$ and its quotient field $K(X)$ are in natural bijective correspondence with the rings between $D$ and $K$; such correspondence is given by $Z\mapsto Z\cap K$, and its inverse by $T\mapsto\kn(T)$. See \cite[\textsection 32]{gilmer} for general properties of the Kronecker function ring of an integral domain.

\section{Results}
We start by showing that the passage from a Pr\"ufer domain $D$ to its Kronecker function ring $\kn(D)$ respects the complete integral closure; this allows to consider only B\'ezout domains instead of general Pr\"ufer domains.
\begin{prop}\label{prop:kn}
Let $D$ be a Pr\"ufer domain. Then, $\kn(D)^\ast=\kn(D^\ast)$.
\end{prop}
\begin{proof}
It is enough to show that $\kn(D)^\ast\cap K=D^\ast$, where $K$ is the quotient field of $D$.

If $z\in D^\ast$, then $cz^n\in D$ for some $c\in D$, $c\neq 0$, for all positive integer $n$; hence, $cz^n\in\kn(D)$ and so $z\in\kn(D)^\ast$.

Conversely, if $z\in\kn(D)^\ast\cap K$, then there is a $c\in\kn(D)$, $c\neq 0$ such that $cz^n\in\kn(D)$ for all positive integer $n$. Since $c\kn(D)\cap D\neq(0)$, we can take $c'\in c\kn(D)\cap D$, $c'\neq 0$. Then $c'z^n\in cz^n\kn(D)\subseteq\kn(D)$; moreover, since $c',z\in K$ we have $c'z^n\in K$. Hence $c'z^n\in\kn(D)\cap K=D$. Thus $z\in D^\ast$.
\end{proof}

The next step is to show that in B\'ezout domains the almost integrality over $D$ of an element only depends on its denominator (when written as a quotient of elements of $D$). We shall use the following terminology.
\begin{defin}
Let $I$ be a finitely generated fractional ideal of the Pr\"ufer domain $D$. The \emph{positive part} of $I$ is
\begin{equation*}
p(I):=I\cap D,
\end{equation*}
while the \emph{negative part} is
\begin{equation*}
n(I):=(I+D)^{-1}=I^{-1}\cap D.
\end{equation*}
\end{defin}

\begin{lemma}\label{lemma:posneg}
Let $D$ be a Pr\"ufer domain and $I$ be a finitely generated fractional ideal of $D$. Then, the following hold.
\begin{enumerate}[(a)]
\item\label{lemma:posneg:id} $p(I)$ and $n(I)$ are invertible ideals contained in $D$.
\item\label{lemma:posneg:decomp} $I=p(I)n(I)^{-1}$.
\item\label{lemma:posneg:coprime} $p(I)+n(I)=D$.
\end{enumerate}
\end{lemma}
\begin{proof}
\ref{lemma:posneg:id} It is clear that $p(I),n(I)\subseteq D$. Moreover, $p(I)$ and $n(I)$ are finitely generated since, in a Pr\"ufer domain, the intersection of two finitely generated ideals is finitely generated; hence they are also invertible.

\ref{lemma:posneg:decomp} follows from the decomposition $I=(I\cap D)(I+D)$ \cite[Theorem 25.2(d)]{gilmer} and the fact that $p(I),n(I)$ are invertible ideals.

\ref{lemma:posneg:coprime} Let $M$ be a maximal ideal containing $p(I)$. Then, $I\cap D\subseteq M$, and thus $(I\cap D)D_M=ID_M\cap D_M\subsetneq D_M$, i.e., $ID_M\subsetneq D_M$ since $D_M$ is a valuation domain. Hence, $I^{-1}D_M=(ID_M)^{-1}\supsetneq D_M$ and so $n(I)D_M=I^{-1}D_M\cap D_M=D_M$, i.e., $n(I)\nsubseteq M$. Therefore, no maximal ideal can contain both $p(I)$ and $n(I)$, and so $p(I)$ and $n(I)$ are coprime.
\end{proof}

\begin{prop}\label{prop:negative}
Let $D$ be a B\'ezout domain with quotient field $K$, and let $z\in K$. Let $n(I)=yD$. Then, $z\in D^\ast$ if and only if $y^{-1}\in D^\ast$.
\end{prop}
\begin{proof}
Let $x:=zy^{-1}$; then, $z=x/y$ and $x$ generates $p(I)$, so that, in particular, $x\in D$. If $y^{-1}\in D^\ast$, then there is a $c\in D$, $c\neq 0$ such that $cy^{-n}\in D$ for all positive integer $n$; hence, $cz^n=cx^n/y^n\in x^nD\subseteq D$ and $z\in D^\ast$. Suppose that $z\in D^\ast$, and let $c\in D$, $c\neq 0$ be such that $cz^n\in D$ for all positive integer $n$. Then, $cx^n\in y^nD$, and so $cx^nD_M\in y^nD_M$ for all maximal ideals $M$ of $D$. We claim that $cy^{-n}\in D$, and it is enough to show that $c\in y^nD_M$ for every $M$. If $x\notin M$, then $cD_M=cx^nD_M\in y^nD_M$. If $x\in M$, then $y\notin M$ since $xD+yD=p(I)+(I)=D$ by Lemma \ref{lemma:posneg}\ref{lemma:posneg:coprime}; thus, $y^nD=D$ and $c\in y^nD_M=D_M$ as well. Hence $y^{-1}\in D^\ast$, as claimed.
\end{proof}

We introduce the topological construction on which our criterion is based.
\begin{defin}
Let $X\subseteq\Spec(D)$. We say that $P\in\Spec(D)$ is a \emph{rim element} for $X$ if $P\notin X$ and there is a $Q\in X$ such that $P\subseteq Q$ and there is no prime ideal properly contained between $P$ and $Q$.

The \emph{rim closure} $\qf(X)$ of $X$ is the union of $X$ and the rim elements of $X$.
\end{defin}

The rim closure is a topological construction, as it only depends uniquely from the Zariski topology.
\begin{lemma}\label{lemma:qf-omef}
Let $D,D'$ be Pr\"ufer domains with a homeomorphism $\phi:\Spec(D)\longrightarrow\Spec(D')$. Then, for every $X\subseteq\Spec(D)$, we have $\phi(\qf(X))=\qf(\phi(X))$.
\end{lemma}
\begin{proof}
We show that $\phi(\qf(X))\subseteq\qf(\phi(X))$. Clearly $\phi(X)\subseteq\qf(\phi(X))$. Let $P$ be a rim element of $x$; then, there is a prime ideal $Q\in X$ such that $P\subseteq Q$ and there are no prime ideals between $P$ and $Q$. Since $P\subseteq Q$ is equivalent to the fact that $Q$ is in the closure of $\{P\}$, it follows that $\phi(P)\subsetneq\phi(Q)$ (as $\phi$ is injective); moreover, if $\phi(P)\subsetneq L\subsetneq \phi(Q)$ for some $L\in\Spec(D')$, then by the same reasoning $P\subsetneq \phi^{-1}(L)\subsetneq Q$, against the fact that $P\in\qf(X)$. Thus $\phi(P)$ must be a rim element of $\phi(X)$, and so $\phi(\qf(X))\subseteq\qf(\phi(X))$.

Applying the previous inclusion to the homeomorphism $\phi^{-1}:\Spec(D')\longrightarrow\Spec(D)$, we have $\phi^{-1}(\qf(\phi(X)))\subseteq\qf(\phi^{-1}\circ\phi(X))=\qf(X)$, i.e., $\qf(\phi(X))\subseteq\phi(\qf(X))$. Thus $\phi(\qf(X))=\qf(\phi(X))$, as claimed.
\end{proof}

\begin{defin}
Let $D$ be a Pr\"ufer domain. We denote by $\Spec^\ast(D)$ the set of all $P\in\Spec(D)$ such that $PD^\ast\neq D^\ast$.
\end{defin}

The following result is the main theorem of the paper: it establishes a topological characterization of the prime ideals in $\Spec^\ast(D)$.
\begin{teor}\label{teor:Specast}
Let $D$ be a Pr\"ufer domain and let $P\in\Spec(D)$. The following are equivalent:
\begin{enumerate}[(i)]
\item\label{teor:Specast:Pin} $P\in\Spec^\ast(D)$;
\item\label{teor:Specast:closed} for every closed cocompact set $X$ of $\Spec(D)$ containing $P$, the rim closure $\qf(X)$ is dense in $\Spec(D)$.
\end{enumerate}
\end{teor}
\begin{proof}
We first consider the case in which $D$ is a B\'ezout domain.

Suppose that $P\in\Spec^\ast(D)$, and let $X$ be a closed cocompact set of $\Spec(D)$ containing $P$; then, $X=\V(xD)$ for some $x\in D$. Suppose that $\qf(X)$ is not dense in $\Spec(D)$; then, there is a closed subset $Y=\V(J)$ containing $\qf(X)$ that is not equal to the whole $\Spec(D)$. In particular, there is a nonzero $c\in J$. We claim that $cx^{-n}\in D$ for every $n$.

Indeed, let $M$ be a maximal ideal of $D$. If $M\notin X$, then $xD_M=D_M$ and so $cx^{-n}D_M=cD_M\subseteq D_M$. Suppose that $M\in X$: since $X$ is cocompact, there is a rim element $Q$ of $X$ that is contained in $M$. Then, $QD_M\subsetneq xD_M$ since $x\notin Q$, and thus $QD_M\subsetneq x^nD_M$ for every $n$. In particular, $c\in x^nD_M$ and so $cx^{-n}\in D_M$. Thus $\displaystyle{cx^{-n}\in\bigcap_{M\in\Max(D)}D_M=D}$ and $x^{-1}$ is almost integral over $D$, i.e., $x^{-1}\in D^\ast$. However, $x\in P$, and thus $PD^\ast=D^\ast$, against the hypothesis $P\in\Spec^\ast(D)$. Therefore, \ref{teor:Specast:closed} holds.

Conversely, suppose that $P\notin\Spec^\ast(D)$; then, $PD^\ast=D^\ast$, and thus (since $D$ is a B\'ezout domain) there is an $x\in P$ such that $xD^\ast=D^\ast$. Hence, we have $x^{-1}\in D^\ast$, and so there is a $c\neq 0$ such that $cx^{-n}\in D$ for all positive integer $n$; we can suppose without loss of generality that $c\in xD$. We claim that $\qf(\V(xD))\subseteq\V(cD)$. Indeed, by construction $\V(xD)\subseteq\V(cD)$. Let $Q$ be a rim element for $\V(xD)$: then, there is a prime ideal $Q'\in\V(xD)$ such that there is no prime ideal between $Q$ and $Q'$. Hence, if $M$ is a maximal ideal of $D$ in $\V(xD)$ we have $\bigcap_nx^nD_M=QD_M$. Since $c\in x^nD$ for every $n$, we have $c\in QD_M$; therefore, $c\in Q$, i.e., $Q\in\V(cD)$ and $\V(cD)$ contains the whole rim closure $\qf(\V(xD))$. Thus, $\qf(\V(xD))$ is not dense.

\medskip

Suppose now that $D$ is a Pr\"ufer (not necessarily B\'ezout) domain. By Proposition \ref{prop:kn}, $P\in\Spec^\ast(D)$ if and only if $P\kn(D)\in\Spec^\ast(\kn(D))$. By the previous part of the proof, since $\kn(D)$ is a B\'ezout domain, $P\kn(D)\in\Spec^\ast(D)$ if and only if condition \ref{teor:Specast:closed} holds in $\kn(D)$. However, by Lemma \ref{lemma:qf-omef} condition \ref{teor:Specast:closed} is invariant by homeomorphisms; since the assignment $P\mapsto P\kn(D)$ is a homeomorphism, the equivalence \ref{teor:Specast:Pin} $\iff$ \ref{teor:Specast:closed} holds also in $D$.
\end{proof}

\begin{cor}
Let $D$ be a Pr\"ufer domain. Then, $D$ is completely integrally closed if and only if the rim closure $\qf(X)$ of every closed cocompact subset $X$ of $\Spec(D)$ is dense in $\Spec(D)$.
\end{cor}
\begin{proof}
If the condition hold, then every prime ideal $P$ of $D$ belongs to $\Spec^\ast(D)$ by Theorem \ref{teor:Specast}, and so $D=D^\ast$. Conversely, if $D=D^\ast$, then $\Spec^\ast(D)=\Spec(D)$, i.e., every prime ideal $P$ belongs to $\Spec^\ast(D)$. The claim follows again from Theorem \ref{teor:Specast}.
\end{proof}

\begin{teor}\label{teor:phiSpecast}
Let $D,D'$ be Pr\"ufer domains such that there is a homeomorphism $\phi:\Spec(D)\longrightarrow\Spec(D')$. Then, $\phi(\Spec^\ast(D))=\Spec^\ast(D')$.
\end{teor}
\begin{proof}
By Lemma \ref{lemma:qf-omef}, condition \ref{teor:Specast:closed} of Theorem \ref{teor:Specast} is invariant by homeomorphisms. Therefore, $\phi$ sends primes of $\Spec^\ast(D)$ inside $\Spec^\ast(D')$, and primes outside $\Spec^\ast(D)$ outside $\Spec^\ast(D')$. The claim follows.
\end{proof}

\begin{cor}
Let $D,D'$ be Pr\"ufer domain with the homeomorphic spectra. Then, $D$ is completely integrally closed if and only if $D'$ is completely integrally closed.
\end{cor}
\begin{proof}
The condition that $D$ is completely integrally closed is equivalent to $\Spec(D)=\Spec^\ast(D)$. The claim follows from Theorem \ref{teor:phiSpecast}.
\end{proof}

We now give a more algebraic version of the condition in Theorem \ref{teor:Specast}.

\begin{teor}\label{teor:MinIMinJ}
Let $D$ be a Pr\"ufer domain and let $P\in\Spec(D)$. The following are equivalent:
\begin{enumerate}[(i)]
\item\label{teor:MinIMinJ:Pin} $P\in\Spec^\ast(D)$;
\item\label{teor:MinIMinJ:Min} if $I,J$ are finitely generated ideals such that $(0)\neq I\subsetneq J\subseteq P$, then $\Min(I)\cap\Min(J)\neq\emptyset$.
\end{enumerate}
\end{teor}
\begin{proof}
We prove that the second condition in the statement is equivalent to condition \ref{teor:Specast:closed} of Theorem \ref{teor:Specast}.

Suppose that the latter holds, and let $I,J$ be finitely generated ideals such that $(0)\neq I\subsetneq J\subseteq P$. Then, $X=\V(J)$ is closed and cocompact, and thus the rim closure $\qf(X)$ is dense in $\Spec(D)$; in particular, $\qf(X)\nsubseteq\V(I)$. Since $\V(J)\subseteq\V(I)$, there is a rim element $Q$ of $X$ that is not in $\V(I)$; by definition, there is a $Q'\in X$ such that $Q\subseteq Q'$ and there are no prime ideals properly contained between $Q$ and $Q'$. Since $\Spec(D)$ is a tree, $Q$ is the only direct predecessor of $Q'$: hence, $Q'$ is a minimal element of $X$, and thus $Q'\in\Min(J)$. Moreover, $Q'$ must also be minimal in $\V(I)$ (otherwise $Q\in\V(I)$); hence $Q'\in\Min(I)\cap\Min(J)$, and in particular $\Min(I)\cap\Min(J)\neq\emptyset$.

Conversely, suppose that condition \ref{teor:MinIMinJ:Min} of the present theorem holds. Let $X$ be a closed cocompact set containing $P$, and let $J$ be a finitely generated ideal such that $X=\V(J)$. If $\qf(X)$ is not dense, then $\qf(X)\subseteq\V(I)$ for some finitely generated ideal $I$; thus $I\subseteq\rad(J)$ and since $I$ is finitely generated $I^n\subseteq J$ for some integer $n$. By hypothesis, $\Min(I^n)\cap\Min(J)$ contains a prime ideal $Q$; moreover, since $Q$ is minimal over a finitely generated ideal there is a prime ideal $Q_0\subsetneq Q$ such that there are no prime ideals properly between $Q_0$ and $Q$. Hence, $Q_0\in\qf(X)\subseteq\V(I)=\V(I^n)$, against the fact that $Q$ is minimal over $I^n$. Therefore $\qf(X)$ is dense and $P\in\Spec^\ast(D)$ by Theorem \ref{teor:Specast}.
\end{proof}

\begin{cor}
Let $D$ be a Pr\"ufer domain. Then, $D$ is completely integrally closed if and only if, whenever $I\subseteq J$ are nonzero proper finitely generated ideals, $\Min(I)\cap\Min(J)\neq\emptyset$.
\end{cor}
\begin{proof}
If $D$ is completely integrally closed, then $\Spec(D)=\Spec^\ast(D)$; if $I\subseteq J$ are nonzero finitely generated proper ideals, then applying Theorem \ref{teor:Specast} to any prime $P$ containing $P$ shows that $\Min(I)\cap\Min(J)\neq\emptyset$. Conversely, if we always have $\Min(I)\cap\Min(J)\neq\emptyset$, then Theorem \ref{teor:Specast} implies that $\Spec(D)=\Spec^\ast(D)$ and so $D=D^\ast$.
\end{proof}

The following result extends \cite[Corollary 8]{gilmer_CIC_1966} from the case where each set $\Min(I)$ is finite to the compact case.
\begin{prop}\label{prop:MinIcomp}
Let $D$ be a Pr\"ufer domain such that, for every finitely generated ideal $I$, the set $\Min(I)$ is compact. Then, $D^\ast$ has dimension at most $1$, and in particular it is completely integrally closed.
\end{prop}
\begin{proof}
Let $P$ be a prime ideal of height at least $2$: then, there is a nonzero finitely generated ideal $I$ contained in $P$ such that $P$ is not minimal over $I$. By \cite[Lemma 9.4]{InvXD}, we can find a finitely generated ideal $J\supseteq I$ such that $\Min(I)\cap\Min(J)\neq\emptyset$. By Theorem \ref{teor:Specast}, $P\notin\Spec^\ast(D)$; hence, $\Spec^\ast(D)$ can contain only $(0)$ and prime ideals of height $1$. Thus $D^\ast$ has dimension at most $1$. The last claim follows from the fact that every Pr\"ufer domain of dimension at most $1$ is completely integrally closed.
\end{proof}

\begin{cor}
Let $D$ be a completely integrally closed Pr\"ufer domain with $\dim(D)>1$. Then, there is a finitely generated ideal $I$ such that $\Min(I)$ is not compact.
\end{cor}
\begin{proof}
Use Proposition \ref{prop:MinIcomp}.
\end{proof}

The previous corollary applies, for example, to the ring $\Int(A)$ of integer-valued polynomials over a Dedekind domain $A$ with finite residue fields, as well as to the ring of entire functions (see, respectively, \cite[Theorem VI.1.7]{intD} and \cite{henriksen_prime}).

\bibliographystyle{plain}
\bibliography{/bib/articoli,/bib/libri,/bib/miei}

\begin{thebibliography}{10}

\bibitem{intD}
Paul-Jean Cahen and Jean-Luc Chabert.
\newblock {\em Integer-{V}alued {P}olynomials}, volume~48 of {\em Mathematical
  Surveys and Monographs}.
\newblock American Mathematical Society, Providence, RI, 1997.

\bibitem{fontana_krr-abRs}
David~E. Dobbs and Marco Fontana.
\newblock Kronecker function rings and abstract {R}iemann surfaces.
\newblock {\em J. Algebra}, 99(1):263--274, 1986.

\bibitem{fontana_libro}
Marco Fontana, James~A. Huckaba, and Ira~J. Papick.
\newblock {\em Pr\"ufer {D}omains}, volume 203 of {\em Monographs and Textbooks
  in Pure and Applied Mathematics}.
\newblock Marcel Dekker Inc., New York, 1997.

\bibitem{gilmer}
Robert Gilmer.
\newblock {\em Multiplicative {I}deal {T}heory}.
\newblock Marcel Dekker Inc., New York, 1972.
\newblock Pure and Applied Mathematics, No. 12.

\bibitem{gilmer_CIC_1966}
Robert~W. Gilmer and William~J. Heinzer.
\newblock On the complete integral closure of an integral domain.
\newblock {\em J. Austral. Math. Soc.}, 6:351--361, 1966.

\bibitem{heinzer-cic}
William Heinzer.
\newblock Some remarks on complete integral closure.
\newblock {\em J. Austral. Math. Soc.}, 9:310--314, 1969.

\bibitem{henriksen_prime}
Melvin Henriksen.
\newblock On the prime ideals of the ring of entire functions.
\newblock {\em Pacific J. Math.}, 3:711--720, 1953.

\bibitem{hill_cic}
Paul Hill.
\newblock On the complete integral closure of a domain.
\newblock {\em Proc. Amer. Math. Soc.}, 36:26--30, 1972.

\bibitem{sheldon-cic}
Philip~B. Sheldon.
\newblock Two counterexamples involving complete integral closure in
  finite-dimensional {P}rufer domains.
\newblock {\em J. Algebra}, 27:462--474, 1973.

\bibitem{InvXD}
Dario Spirito.
\newblock Radical factorization in higher dimension.
\newblock submitted, arXiv:2409.10219.

\end{thebibliography}
\end{document}